\documentclass{endm}
\usepackage{endmmacro}
\usepackage{graphicx}

\begin{document}

\begin{verbatim}\end{verbatim}\vspace{2.5cm}

\begin{frontmatter}

\title{A graph-theoretical axiomatization of oriented matroids}

\author{Kolja Knauer\thanksref{myemail}}
\address{Institut f\"ur Mathematik\\ Technische Universit\"at Berlin\\ Berlin, Germany}
\author{Juan Jos\'e Montellano-Ballesteros\thanksref{coemail}}
\author{, Ricardo Strausz\thanksref{cocoemail}}
\address{Instituto de Matem\'aticas\\ Universidad Nacional Aut\'onoma de M\'exico\\ M\'exico D.F., M\'exico} 
\thanks[myemail]{Email:\href{mailto:knauer@math.tu-berlin.de} {\texttt{\normalshape knauer@math.tu-berlin.de}}}
\thanks[coemail]{Email:\href{mailto:juancho@matem.unam.mx} {\texttt{\normalshape juancho@matem.unam.mx}}}
\thanks[cocoemail]{Email:\href{mailto:strausz@math.unam.mx} {\texttt{\normalshape strausz@math.unam.mx}}}

\begin{abstract}
We characterize which systems of sign vectors are the cocircuits 
of an oriented matroid in terms of the cocircuit graph.
\end{abstract}


\end{frontmatter}

\section{Introduction}
The cocircuit graph is a natural combinatorial object associated with an oriented
matroid. In the case of spherical pseudoline-arrangements, i.e., rank 3 oriented
matroids, its vertices are the intersection points of the lines and two points
share an edge if they are adjacent on a line. More generally, the Topological
Representation Theorem of Folkman and Lawrence~\cite{Fol-78} says that every oriented
matroid can be represented as an arrangement of pseudospheres. The cocircuit
graph is the 1-skeleton of this arrangement.
Cordovil, Fukuda and Guedes de Oliveira~\cite{Cor-93} show that a a cocircuit graph
does not uniquely determine an oriented matroid. But Babson, Finschi and Fukuda~\cite{Bab-01}
show that \emph{uniform} oriented matroids are determined up to isomorphism by their cocircuit graph.
Moreover they provide a polynomial time recognition algorithm
for cocircuit graphs of uniform oriented matroids.
In~\cite{Mon-06}, Montellano-Ballesteros and Strausz give a characterization of uniform
oriented matroids in view of sign labeled cocircuit graphs. This characterization is 
strengthened by Felsner, G\'omez, Knauer, Montellano-Ballesteros and Strausz~\cite{Fel-11} and used 
to improve the recognition algorithm of~\cite{Bab-01}.

In this paper we present a generalization and strengthening of the characterization of sign labeled 
cocircuit graphs of uniform oriented matroids of~\cite{Fel-11} to general oriented matroids. After introducing the 
necessary preliminaries in the next section, we prove the main theorem in the last section.

\section{Preliminaries}

Here we will only introduce the terminology necessary for proving our result, for a more general introduction, see~\cite{Bjo-99}. A \emph{signed set} $X$ on a ground set $E$ is pair $X=(X^+,X^-)$ of disjoint subsets of $E$. For $e\in E$ we write $X(e)=+$ and $X(e)=-$
if $e\in X^+$ and $e\in X^-$, respectively, and $X(e)=0$, otherwise. The \emph{support}
$\underline{X}$ of a signed set $X$ is the set $X^+\cup X^-$. The \emph{zero-support} of $X$ is $X^0:=E\backslash\underline{X}$. By $-X$ we denote the signed set $(X^-,X^+)$. Given
signed sets $X,Y$ their \emph{separator} is defined as $S(X,Y):=(X^+\cap Y^-)\cup(X^-\cap Y^+)$.

\begin{definition}
A pair $\mathcal{M}=(E, \mathcal{C}^*)$ is called \emph{oriented matroid} with \emph{cocircuits} $\mathcal{C}^*$ if $\mathcal{C}^*$ is a system of signed sets with ground set $E$, satisfying the following axioms: 
\begin{list}{\labelitemi}{\leftmargin=2.5em}
 \item[(C0)] $\emptyset\notin\mathcal{C}^*$
 \item[(C1)] $\mathcal{C}^*=-\mathcal{C}^*$
 \item[(C2)] if $X,Y\in\mathcal{C}^*$ and $\underline{X}\subseteq\underline{Y}$ then $X=\pm Y$
 \item[(C3)] for all $X,Y\in\mathcal{C}^*$ with $X\neq\pm Y$ and $e\in S(X,Y)$ exists $Z\in\mathcal{C}^*$ with $Z(e)=0$, $Z^+\subseteq X^+\cup Y^+$ and $Z^-\subseteq X^-\cup Y^-$
\end{list}
\end{definition}

The \emph{composition} of signed sets $X,Y$ is the signed set $X\circ Y:=(X^+\cup(Y^+\backslash X^-),X^-\cup(Y^-\backslash X^+)).$
Given a system $\mathcal{C}^*$ of signed sets we denote by $\mathcal{L}(\mathcal{C}^*):=\{X_1\circ\ldots\circ X_k\mid X_1, \ldots, X_k\in \mathcal{C}^*\}$ the set of all (finite) compositions of $\mathcal{C}^*$. Note that the empty set is considered as the empty composition of cocircuits, and so $\emptyset\in \mathcal{L}(\mathcal{C}^*)$. If $\mathcal{C}^*$ are the cocircuits of an oriented matroid $\mathcal{M}$, then the elements of $\mathcal{L}(\mathcal{C}^*)$ are called the \emph{covectors} of $\mathcal{M}$. One can endow $\mathcal{L}(\mathcal{C}^*)$ with a partial order relation where $Y\leq X$ if and only if $S(X,Y)=\emptyset$ and $\underline{Y}\subseteq \underline{X}$. Adding a global maximum $\hat{\mathcal{L}}:=\mathcal{L}(\mathcal{C}^*)\cup\hat{1}$ it is easy to see that one obtains a lattice $\mathcal{F}_{{\rm big}}(\mathcal{L}):=(\hat{\mathcal{L}},\leq)$. If $\mathcal{C}^*$ is the set of cocircuits of an oriented matroid, then $\mathcal{F}_{{\rm big}}(\mathcal{L})$ is graded lattice with rank function $h$. In this case $\mathcal{F}_{{\rm big}}(\mathcal{L})$ is called \emph{the big face lattice} of $\mathcal{M}$. The \emph{rank} $r(\mathcal{M})$ of $\mathcal{M}$ is $h(\hat{1})-1$, i.e., one less than the rank of $\mathcal{F}_{{\rm big}}(\mathcal{L})$.

There are two important undirected graphs associated to $\mathcal{F}_{{\rm big}}(\mathcal{L})$ -- one on its atoms and one on its coatoms. So the first is a graph on $\mathcal{C}^*$. In the case of $\mathcal{C}^*$ being the cocircuits
of an oriented matroid $\mathcal{M}$ it is called the \emph{cocircuit graph} of $\mathcal{M}$. Define $G(\mathcal{C}^*)$ on $\mathcal{C}^*$ such that two signed sets $X,Y\in\mathcal{C}^*$ are connected by an edge if and only if there is $Z\in\hat{\mathcal{L}}$ such that $X,Y$ are the only elements of $\mathcal{C}^*$ with $X,Y\leq Z$. 

The other graph induced by $\mathcal{F}_{{\rm big}}(\mathcal{L})$ is defined on the set $\mathcal{T}$ of coatoms of $\mathcal{F}_{{\rm big}}(\mathcal{L})$ the poset. Elements of $\mathcal{T}$ are called \emph{topes}. Topes $S,T\in\mathcal{T}$ are contained in an edge 
if and only if there is $Z\in\hat{\mathcal{L}}$ such that $S,T$ are the only elements of $\mathcal{T}$ with $X,Y\geq Z$.
This graph called the \emph{tope graph} is denoted by $G(\mathcal{T})$.

If $G$ is a graph on a system $\mathcal{S}$ of signed sets with ground set $E$. For $X_1,\ldots, X_k\in \mathcal{S}$ we denote by $[X_1,\ldots, X_k]$ the subgraph of $G$ induced by $\{Z\in \mathcal{S}\mid Z(e)\in\{0,X_1(e),\ldots, X_k(e)\}\textmd{ for all }e\in E\}$. We call $[X_1,\ldots, X_k]$ the \emph{crabbed hull} of $X_1,\ldots, X_k$. An $(X,Y)$-path in $G$ is called \emph{crabbed} if it is contained in $[X,Y]$.

One important oriented matroid operation is the \emph{contraction}. Let $A\subseteq E$, then $\mathcal{M}/A$ is an oriented matroid
on the ground set $E\backslash A$ with $\mathcal{C}^*/A:=\{X\backslash A\mid X\in\mathcal{C}^*\textmd{ and }A\subseteq X^0\}$. The set $\mathcal{L}(\mathcal{C}^*/A)$ is easily seen to be $\{X\backslash A\mid X\in\mathcal{L}(\mathcal{C}^*)\textmd{ and }A\subseteq X^0\}$. It is easy to see that for $U\in\mathcal{L}(\mathcal{C}^*)$ we have $r(\mathcal{M}/U^0)=h(U)$, where $h(U)$ is the rank of $U$ in $\mathcal{F}_{{\rm big}}(\mathcal{L})$.

\section{Result}

In order to prove Theorem~\ref{thm:main} we need two lemmas. The first one is about tope graphs of oriented matroids.
Tope graphs of oriented matroids are a special class of \emph{antipodal partial cubes}~\cite{Fuk-93}. We will make
use of a particular consequence of this:
\begin{lemma}[\cite{Cor-82}]\label{lem:tope1}
 Let $\mathcal{M}$ be an oriented matroid with topes $U,V\in \mathcal{T}$. For all $U,V\in\mathcal{T}$ there is a crabbed $(U,V)$-path in 
$G(\mathcal{T})$.
\end{lemma}

The second lemma establishes a connection between tope graph and cocircuit graph. As an application of a theorem of Barnette~\cite{Bar-73} Cordovil and Fukuda prove:
\begin{lemma}[\cite{Cor-93}]\label{lem:tope2}
Let $\mathcal{M}$ be an oriented matroid of rank $r$ and $U\in \mathcal{T}$ a tope of $\mathcal{M}$. The graph $G(U)$ induced by $\{X\in \mathcal{C}^*\mid X\circ U=U\}$ in $G(\mathcal{C}^*)$ is $(r-1)$-connected.
\end{lemma}

Together this enables us to prove a graph-theoretical axiomatization of oriented matroids:
\begin{theorem}\label{thm:main}
Let $\mathcal{C}^*$ be a set of sign vectors satisfying (C0)--(C2) then the following are equivalent

\begin{list}{\labelitemi}{\leftmargin=2.5em}
 \item[(i)] $\mathcal{C}^*$ is the set of cocircuits of an oriented matroid $\mathcal{M}$,
\item[(ii)] the crabbed hull $[X_1,\ldots, X_k]$ of any $X_1,\ldots, X_k\in\mathcal{C}^*$ is an induced subgraph of connectivity $h(X_1\circ\ldots\circ X_k)-1$ of $G(\mathcal{C}^*)$,
\item[(iii)] for all $X, Y\in\mathcal{C}^*$ with $X\neq\pm Y$ there is a crabbed $(X,Y)$-path in $G(\mathcal{C}^*)$.
\end{list}
\end{theorem}
\begin{proof}
(i)$\Longrightarrow$ (ii): Let $U:=X_1\circ\ldots\circ X_k$ be a covector of rank $r':=h(X_1\circ\ldots\circ X_k)$ and $X,Y$ cocircuits in $[X_1,\ldots, X_k]$. Contract $U^0$ obtaining $\mathcal{M}':=\mathcal{M}/U^0$ of rank $r'$ and cocircuits $X',Y'$. The contraction does not affect the crabbed hull we are considering, i.e. $[X_1,\ldots, X_k]\cong [X'_1,\ldots, X'_k]$. Now $U'$ is a tope of $\mathcal{M}'$ and so are $V':=X'\circ U'$ and $W':=Y'\circ U'$. By Lemma~\ref{lem:tope1} there is a crabbed $(V',U')$-path $P=(V'=T_1,\ldots, T_k=U')$ in $G(\mathcal{T}')$. The graphs $G(T_i)$ are all contained in $[X'_1,\ldots, X'_k]$ and $(r-1)$-connected by Lemma~\ref{lem:tope2}. Consecutive $G(T_i)$ and $G(T_{i+1})$ intersect in at least $r'-1$ vertices, because their intersection is a tope of a one-element-contraction minor of $\mathcal{M}'$. Together Menger's theorem (see e.g.~\cite{Die-10}) yields that the graph $G(T_1)\cup\ldots\cup G(T_k)$ is $(r'-1)$-connected. In particular there are $(r'-1)$ internally disjoint paths connecting $X'$ and $Y'$ in $[X'_1,\ldots, X'_k]$ and thus the analogue holds for $X$ and $Y$ in $[X_1,\ldots, X_k]$. Hence 
$[X_1,\ldots, X_k]$ is $(h(X_1\circ\ldots\circ X_k)-1)$-connected.

(ii)$\Longrightarrow$ (iii): If $X\neq\pm Y$ then $(h(X\circ Y)-1)>0$. Hence $[X,Y]$ is connected and there is a crabbed $(X,Y)$-path in $G(\mathcal{C}^*)$.

(iii)$\Longrightarrow$ (i): We have to show that (C3) holds for $\mathcal{C}^*$. Let $X, Y\in\mathcal{C}^*$ with $X\neq\pm Y$
and $e\in S(X,Y)$. Let $P$ be a crabbed $(X,Y)$-path. Since adjacent cocircuits have empty separator, there must be $Z\in P$ with
$Z(e)=0$. Since $P$ is crabbed $Z$ also satisfies $Z^+\subseteq X^+\cup Y^+$ and $Z^-\subseteq X^-\cup Y^-$.
\end{proof}

It shall be mentioned that the ``(i)$\Longrightarrow$ (ii)''-part of the proof is only a slight generalization of a result in~\cite{Cor-93}. But there the characterizing quality of (ii) was not noted. Furthermore we remark that the connectivity in (ii) is best-possible, since in uniform oriented matroids $X_i$ has exactly $h(X_1\circ\ldots\circ X_k)-1$ neighbors in $[X_1,\ldots, X_k]$.

Even if the cocircuit graph does not uniquely determine the oriented matroid, Theorem~\ref{thm:main} might lead to an effective recognition algorithm for cocircuit graphs of general oriented matroids, as its uniform specialization did in~\cite{Fel-11}. In particular, if one is given $G(\mathcal{C}^*)$  with edge set $\mathcal{E}$ it is possible to check (iii) in $\mathcal{O}(|\mathcal{C}^*||\mathcal{E}|)$, see part 5.C. of the algorithm in~\cite{Fel-11}. In contrast the naive algorithm to
check (C3) takes $\mathcal{O}(|\mathcal{C}^*|^3)$. So we have an advantage for sparse cocircuit graphs, e.g., cocircuit graphs of uniform oriented matroids.

Another goal would surely be to characterize cocircuit graphs in purely graph-theoretic terms, i.e., excluding any information about signed sets at all.

\end{document}